\documentclass[10pt,reqno]{amsart}
\usepackage{amssymb,amsfonts}
\usepackage[all,arc]{xy}
\usepackage{enumerate}
\usepackage{mathrsfs}
\usepackage{geometry}
\usepackage{amsmath}
\usepackage{latexsym, bm}
\usepackage{indentfirst}
\usepackage{CJK}
\usepackage{amsthm}
\usepackage{hyperref}
\usepackage{cite}
\usepackage{lipsum}
\usepackage[misc]{ifsym} 

\newcommand{\vp}{\varphi}
\newcommand{\pa}{\partial}

\newcommand{\na}{\nabla}
\newcommand{\al}{\alpha}

\newcommand{\tl}{\tilde}
\newcommand{\wtl}{\widetilde}
\newcommand{\ov}{\overline}
\newcommand{\la}{\lambda}

\newcommand{\wg}{\wedge}

\newcommand{\bi}{\bar{i}}
\newcommand{\bj}{\bar{j}}
\newcommand{\bl}{\bar{l}}

\newcommand{\de}{\delta}

\newtheorem{thm}{Theorem}[section]
\newtheorem{cor}[thm]{Corollary}
\newtheorem{lem}[thm]{Lemma}
\newtheorem{prop}[thm]{Proposition}

\theoremstyle{definition}
\newtheorem{defn}[thm]{Definition}
\newtheorem{remk}[thm]{\normalfont{\textit{Remark}}}

\numberwithin{equation}{section}

\begin{document}
\title{Compact Hermitian surfaces with pointwise constant Gauduchon holomorphic sectional curvature}

\author{Haojie Chen, Xiaolan Nie}
\address{Department of Mathematics, Zhejiang Normal University, Jinhua Zhejiang, 321004, China}
\email{chj@zjnu.edu.cn, \ nie@zjnu.edu.cn }
\date{}
\maketitle

\begin{abstract} Motivated by a recent work of Chen-Zheng\cite{CZ} on Strominger space forms, we prove that  a compact Hermitian surface with pointwise constant holomorphic sectional curvature with respect to a Gauduchon connection $\nabla^t $ is either  K\"ahler, or an isosceles Hopf surface with an admissible metric and $t=-1$ or $t=3$. In particular,  a compact Hermitian surface with pointwise constant Lichnerowicz holomorphic sectional curvature is K\"ahler. We further generalize the result to the case for the two-parameter canonical connections introduced by Zhao-Zheng\cite{ZZ2}, which extends  a previous result by Apostolov-Davidov-Mu\v{s}karov\cite{ADM}.\\



\end{abstract}

{\let\thefootnote\relax\footnotetext{This research is partially supported by NSFC grants No. 11901530, No. 11801516, No. 12071161 and Zhejiang Provincial NSF grant No. LY19A010017.}}

\section{Introduction}
 This note mainly concerns Hermitian surfaces with pointwise constant holomorphic sectional curvature with respect to a Gauduchon connection $\na^t$. Given a Hermitian manifold $(M,J,g)$, there is a one-parameter family of canonical Hermitian connections $\nabla^t$ introduced by Gauduchon \cite{G} in 1997, which is the line joining the Chern connection $\nabla^c$, the Strominger connection $\nabla^s$ (also known as Bismut connection) and the Lichnerowicz connection $\na^l$. The last one is the restriction of the Levi-Civita connection $\nabla^{LC}$ on the holomorphic tangent bundle $T^{1,0}M$ (also known as the associated connection \cite{GK} or the induced Levi-Civita connection \cite{LY}). For any $t\in \mathbb R$, the Gauduchon connection is defined by $$\nabla^t=\frac{1+t}{2}\nabla^c+\frac{1-t}{2}\nabla^s=t\na^c+(1-t)\na^l.$$In particular, $\nabla^1=\nabla^c$, $\nabla^{-1}=\nabla^s$ and $\nabla^0=\nabla^l$  (see also \cite{ZZ2}). If $g$ is K\"ahler, then all $\nabla^t$ coincide with the Levi-Civita connection. Otherwise, they are mutually different. While each Gauduchon connection exhibits its own geometry, the properties of the whole connection family may reflect the intrinsic features of the Hermitian manifold. 

There has been many recent studies on the curvatures of special connections on Hermitian manifolds, e.g. for the Chern connection (\cite{B},\cite{BG},\cite{CCN},\cite{LU},\cite{LZ},\cite{RZ},\cite{T}), the Levi-Civita connection (\cite{ADM},\cite{KYZ},\cite{SS}), the Strominger connection  (\cite{CZ},\cite{WYZ},\cite{YZZ},\cite{ZZ19}) and the Lichnerowicz connection (\cite{GK},\cite{HLY},\cite{LY2}). There are also some recent work on the curvatures of general Gauduchon connections (see \cite{AOUV},\cite{VYZ},\cite{WY},\cite{YZ2},\cite{ZZ2}\ etc). Among various curvature notions, the holomorphic sectional curvature is a natural substitute of sectional curvature on Hermitian manifolds (see e.g. \cite{KN},\cite{Z} and the references therein). A fundamental question is to understand Hermitian manifolds with constant or pointwise constant holomorphic sectional curvature. In 1985, Balas-Gauduchon \cite{BG} prove that a compact Hermitian surface with constant nonpositive holomorphic sectional curvature with respect to $\na^c$ must be K\"ahler. This result was beautifully extended by Apostolov-Davidov-Mu\v{s}karov \cite{ADM} in 1996. They prove the following theorem, as an application of their classification of compact self-dual Hermitian surfaces. 
\begin{thm}[Apostolov-Davidov-Mu\v{s}karov] \label{thm1}
Any compact Hermitian surface with point-wise constant holomorphic sectional curvature with respect to the Levi-Civita connection or the Chern connection must be K\"ahler.
\end{thm} 
Recently,  an important new progress is made by  S. Chen and F. Zheng in  \cite{CZ}, where they study the (weak) Srominger space forms, i.e., compact Hermitian manifolds with (pointwise) constant Strominger holomorphic sectional curvature. In particular, they give a complete classification in dimension two and prove that
 \begin{thm}[Chen-Zheng]
Any compact Hermitian surface $(M, g)$ with pointwise constant Strominger holomorphic sectional curvature must be either K\"ahler or an isosceles Hopf surface with an admissible metric.
\end{thm} 
Motivated by the above results, we study Hermitian surfaces with pointwise constant holomorphic sectional curvature with respect to an arbitrary Gauduchon connection. Using similar techniques as in \cite{CZ}, we prove the following theorem. 
\begin{thm} \label{thm0}Let $(M, g)$ be a compact Hermitian surface with pointwise constant holomorphic sectional curvature with respect to a Gauduchon connection $\nabla^t$. Then either (i) $g$ is K\"aher, or (ii) $(M,g)$ is an isosceles Hopf surface with an admissible metric and in this case $t=-1$ or  $t=3$.
\end{thm}
Note that for a K\"ahler manifold of complex dimension $n\geq 2$,  by the Schur's lemma, if the holomorphic sectional curvature is pointwise constant, it must be constant. A complete K\"ahler manifold with constant holomorphic sectional curvature is called a complex space form.
It is known that a simply connected complex space form is holomorphically isometric to $\mathbb {CP}^n, \mathbb C^n$ or $\mathbb {CH}^n$ (\cite{Boo},\cite{Ha}). Therefore, in case (i) of Theorem \ref{thm0}, $(M,g)$ is isomorphic to one of the following: $\mathbb {CP}^2$, complex torus, hyperelliptic surface and a compact quotient of $\mathbb {CH}^2$. Admissible metrics on an isosceles Hopf surface are a special class of Hermitian metrics which are conformal to the standard Hopf metric (see section 5 for more details).
\begin{remk}
The case for the Chern connection ($t=1$) is proved by Apostolov-Davidov-Mu\v{s}karov \cite{ADM} and the case for the Strominger connection ($t=-1$ ) is proved by Chen-Zheng \cite{CZ}. 
\end{remk}
As $\nabla^0=\nabla^l$, we have a direct corollary.
\begin{cor} Any compact Hermitian surface with pointwise constant holomorphic sectional curvature with respect to the Lichnerowicz connection is K\"ahler.
\end{cor}
From the proof of Theorem \ref{thm0}, it follows that an admissible metric on isosceles Hopf surfaces has the same pointwise constant holomorphic sectional curvatures with respect to $\nabla^s$ and $\nabla^3$. However, the curvature tensors of $\nabla^s$ and $\nabla^3$ are different. In particular, the standard Hopf metric has zero curvature tensor for $\nabla^s$ and nonzero curvature tensor for $\nabla^3$. This gives an example of compact Hermitian surfaces with zero Gauduchon holomorphic sectional curvature and nonvanishing curvature tensor (Corollary \ref{cor5.4}).

Next, we put the Levi-Civita connection into consideration. As is known, if $g$ is not K\"ahler, the Levi-Civita connection $\nabla^{LC}$ is not a Hermitian connection and does not coincide with any $\nabla^t$. As introduced by Zhao-Zheng in \cite{ZZ2}, define the two-parameter canonical connections $$D^t_s=(1-s)\nabla^t+s\nabla^{LC}$$ for any $(t, s)\in \mathbb R^2$. It is the connection plane spanned by the Levi-Civita connection and the one-parameter Gauduchon connections $\nabla^t$. Note that $D^t_1\equiv\nabla^{LC}$ and $D^t_0=\nabla^t$. For a non-K\"ahler metric $g$ and $s\neq 0$, $D^t_s$ is not a Hermitian connection either. We obtain the following result which is a generalization of Theorem \ref{thm0}.
\begin{thm}\label{thm2} Let $(M, g)$ be a compact Hermitian surface with pointwise constant holomorphic sectional curvature with respect to a canonical $(t,s)$ connection $D^t_s$. Then either (i) $g$ is K\"aher, or (ii) $(M,g)$ is an isosceles Hopf surface with an admissible metric and in this case $(1-t+ts)^2+s^2=4$.
\end{thm}
\begin{remk}
When $s=1$, $D_1^t$ is the Levi-Civita connection. In this case, the result is proved in \cite{ADM}.
\end{remk}
As a consequence of Theorem \ref{thm2}, we prove the following characterizations.
\begin{cor} \label{cor1.8}
Let $(M, g)$ be a compact Hermitian surface. If one of the following is satisfied,
\begin{itemize}
\item[(i)] $g$ has positive or negative pointwise constant holomorphic sectional curvature with respect to $D^t_s$ for some $(t,s)\in \mathbb R^2$,
\item[(ii)] $g$ has zero holomorphic sectional curvature with respect to $D^t_s$ for $(t,s)$ satisfying $(1-t+ts)^2+s^2\neq 4$,
\end{itemize}
 then $(M,g)$ is K\"ahler and a complex space form.
\end{cor}
On the other side, if $g$ is non-K\"ahler and has zero holomorphic sectional curvature with respect to some $D^t_s$, then it must be a scalar multiple of the standard metric on an isosceles Hopf surface and $(1-t+ts)^2+s^2=4$.

A key step to prove Theorem \ref{thm0} and Theorem \ref{thm2} is to show that a Hermitian surface with pointwise constant holomorphic sectional curvature with respect to $\nabla^t$ or $D^t_s$ must be a self-dual Riemannian 4-manifold. We show that it also holds for general almost Hermitian surfaces (Proposition \ref{prop2.1}). A natural question is to classify compact almost Hermitian surfaces with constant or pointwise constant Gauduchon holomorphic sectional curvature.

 In section 5, we also study Hermitian metrics with pointwise constant Gauduchon holomorphic sectional curvatures on higher dimensional isosceles Hopf manifolds and prove the following (Proposition 5.3).
\begin{prop} \label{prop1.7}
A Hermitian metric $\omega$ on an isosceles Hopf manifold which is conformal to the standard Hopf metric has pointwise constant holomorphic sectional curvature with respect to $D^t_s$ if and only if $(1-t+ts)^2+s^2=4$ and $\omega$ is an admissible metric.  \end{prop}
The structure of the paper is as follows. In section 2, we briefly introduce the canonical connections and holomorphic sectional curvature on an almost Hermitian manifold. In section 3, we show that an almost Hermitian surface with pointwise constant holomorphic sectional curvature must be self-dual. In section 4, we compute the transformation of curvature components of the canonical connections $\nabla^t$ and $D^t_s$ under conformal change of Hermitian metrics on a complex manifold. In section 5, we study Hermitian metrics  with pointwise constant holomorphic sectional curvature on an isosceles Hopf manifold. In section 6, we prove Theorem \ref{thm2} and Corollary \ref{cor1.8}. Throughout the paper, we
 assume $M$ to be connected.\\

\noindent \textbf{Acknowlegments.} We are deeply indebted to Professor Fangyang Zheng for many stimulating and valuable discussions and particularly to the work \cite{CZ}. We are also very grateful to Professors Kefeng Liu and Jiaping Wang for helpful communications and their encouragement.

\section{Canonical connections}
In this section, we give some necessary notations. We refer the readers to\cite{CZ} \cite{FZ} \cite{YZ1} for more details. Let $(M,J,g)$ be an almost Hermitian manifold. 
Comparing to the Levi-Civita connection $\nabla^{LC}$ which preserves $g$ and has vanishing torsion, the Chern connection is the unique connection $\nabla^c$ which satisfies $\nabla^cg=0, \nabla^cJ=0$ and has vanishing $(1,1)$ part of the torsion. Fix any $q\in M$. Let $\{e_1, e_2, ..., e_n\} $ be a local unitary frame of $(1,0)$ vectors around $q$ with $\{\vp^1, \vp^2, ..., \vp^n\}$ being the dual frame. Assume that \begin{align} \label{2.00} \nabla^{LC}e_i=(\theta_1)^j_ie_j+(\ov{\theta_2})^j_i\bar{e}_j, \ \ \ \nabla^ce_i=\theta^j_ie_j, \end{align}
where $\theta_1, \theta_2, \theta$ are the matrices of connection one-forms and the Einstein summation notation is used. Denote the torsion forms of the Chern connection by $$\tau^i=T^i_{jk}\varphi^j\wedge\varphi^k+T^i_{\bar{j}\bar{k}}\ov{\varphi^j}\wedge\ov{\varphi^k},$$ with $T^i_{jk}=-T^i_{kj}, T^i_{\bar{j}\bar{k}}=-T^i_{\bar{k}\bar{j}}$. It is known that $T^i_{\bar{j}\bar{k}}=\frac{1}{2}N^i_{\bar{j}\bar{k}}$, where $$N^i_{\bar{j}\bar{k}}=-g( [\bar{e}_j,\bar{e}_k],\ov{e_i})$$ are the components of the Nijenhuis tensor. For convenience, here  $T^i_{\bar{j}\bar{k}}$  and $N^i_{\bar{j}\bar{k}}$ may be different with the notations in some literature up to a scalar. So if $J$ is integrable, then $T^i_{\bar{j}\bar{k}}=0$. The following relations between $\nabla^{LC}$ and $\nabla^c$ are known ( see e.g. \cite{FZ}\cite{YZ1}). 
\begin{align}
\gamma^j_i&=(\theta_1)^j_i-\theta^j_i=T_{ik}^j\vp^k-\ov{T^i_{jk}}\ov{\vp^k}, \label{2.1}\\
(\theta_2)^j_i&=\ov{T^k_{ij}}\vp^k+(T^j_{\bar{i}\bar{k}}+T^i_{\bar{k}\bar{j}}+T^k_{\bar{i}\bar{j}})\ov{\vp^k}. \label{2.6}
\end{align}
For a connection $\nabla$, denote the curvature tensor to be
\begin{align*} 
&R^{\nabla}(X,Y)Z=\nabla_X\nabla_YZ-\nabla_Y\nabla_XZ-\nabla_{[X,Y]}Z,\\
&R^{\nabla}(X,Y,Z,W)=g(R^{\nabla}(X,Y)Z,W)
\end{align*}
for any $X,Y,Z,W\in TM\otimes \mathbb C$. In particular, denote $R$ and $R^c$ to be the curvature tensor of the Levi-Civita connection and the Chern connection. Assume that $\Theta_1, \Theta$ are the matrices of  curvature 2-forms of $\nabla^{LC}, \nabla^c$ with items $(\Theta_1)^j_i(X,Y)=R(X,Y,e_i,\bar{e}_j)$ and $\Theta^j_i(X,Y)=R^c(X,Y,e_i,\bar{e}_j)$. From the structure equations (see e.g. \cite{YZ1}), we have: \begin{align}
(\Theta_1)^j_i&=d(\theta_1)^j_i-(\theta_1)_i^k\wedge(\theta_1)_k^j-(\ov{\theta_2})^k_i\wedge(\theta_2)_k^j,\label{2.01}\\
\Theta^j_i&=d\theta^j_i-\theta_i^k\wedge\theta_k^j.\label{2.02}
\end{align}
Then by (\ref{2.1}), \begin{align}(\Theta_1)^j_i-\Theta^j_i=d\gamma^j_i-\gamma^k_i\wedge\gamma_k^j-\gamma^k_i\wedge\theta_k^j-\theta^k_i\wedge\gamma_k^j-(\ov{\theta_2})^k_i\wedge(\theta_2)_k^j.\label{2.22} \end{align}
The Lichnerowicz connection $\nabla^l$ is the restriction of $\nabla^{LC}$ to $T^{1,0}M$. Therefore, we have $$\nabla^le_i=(\theta_1)^j_ie_j.$$
\begin{defn}
For any $t\in \mathbb R$, the Gauduchon connection $\nabla^t$ is defined to be $$\nabla^t=t\nabla^c+(1-t)\nabla^l.$$
\end{defn}
 Assume that $\nabla^te_i=(\theta^t)^j_ie_j$. Then \begin{align}\label{2.5}
(\theta^t)_i^j=t\theta^j_i+(1-t)(\theta_1)^j_i=\theta^j_i+(1-t)\gamma^j_i.\end{align}
Denote $\Theta^t$ to be the matrix of curvature 2-forms of $\nabla^t$. Then 
\begin{align} \label{2.06}
(\Theta^t)^j_i=d(\theta^t)^j_i-(\theta^t)^k_i\wedge(\theta^t)_k^j.
\end{align}
Therefore, by (\ref{2.02}), (\ref{2.5}) and (\ref{2.06})
\begin{align}
(\Theta^t)^j_i-\Theta^j_i=(1-t)d\gamma^j_i-(1-t)\gamma^k_i\wedge\theta_k^j-(1-t)\theta^k_i\wedge\gamma_k^j-(1-t)^2\gamma^k_i\wedge\gamma_k^j. \label{2.3}
\end{align}
Then by (\ref{2.22}) and (\ref{2.3})
\begin{align}
(\Theta^t)^j_i-(\Theta_1)^j_i=-td\gamma^j_i+t\gamma^k_i\wedge\theta_k^j+t\theta^k_i\wedge\gamma_k^j+(2t-t^2)\gamma^k_i\wedge\gamma_k^j+(\ov{\theta_2})^k_i\wedge(\theta_2)_k^j. \label{2.07}
\end{align}
Note that $R^t_{k\bar{l}i\bar{j}}=(\Theta^t)^j_i(e_k,\bar{e}_l), R_{k\bar{l}i\bar{j}}=(\Theta_1)^j_i(e_k,\bar{e}_l)$. Then by (\ref{2.07}) we get the following lemma which has been obtained in Proposition 4.2 in \cite{FZ}.
\begin{lem}\label{lem2.01}
\begin{align*} 
R^t_{k\bar{l}i\bar{j}}&=R_{k\bar{l}i\bar{j}}+t(T^j_{ik,\bar{l}}+\ov{T^i_{jl,\bar{k}}})+(t^2-2t)(T^r_{ik}\ov{T^r_{jl}}-T^j_{rk}\ov{T^i_{rl}})\\
&-\ov{T^k_{rj}}T^l_{ir}+(T^j_{\bar{r}\bar{l}}+T^r_{\bar{l}\bar{j}}+T^l_{\bar{r}\bar{j}})(\ov{T^r_{\bar{i}\bar{k}}}+\ov{T^i_{\bar{k}\bar{r}}}
+\ov{T^k_{\bar{i}\bar{r}}})
\end{align*}
If $J$ is integrable, then  \begin{align*} 
R^t_{k\bar{l}i\bar{j}}&=R_{k\bar{l}i\bar{j}}+t(T^j_{ik,\bar{l}}+\ov{T^i_{jl,\bar{k}}})+(t^2-2t)(T^r_{ik}\ov{T^r_{jl}}-T^j_{rk}\ov{T^i_{rl}})-\ov{T^k_{rj}}T^l_{ir}
\end{align*} Here the subscripts in $T$ stands for covariant derivatives with respect to $\nabla^c$. 
\end{lem}
Next, we introduce the two-parameter connections studied by Zhao-Zheng in \cite{ZZ2}. It is the connection plane spanned by the Levi-Civita connection and the Gauduchon connections. 
\begin{defn}
For any $(t,s)\in \mathbb R^2$, define the $(t,s)$ canonical connection to be $$D^t_s=(1-s)\nabla^t+s\nabla^{LC}.$$
\end{defn}
By (\ref{2.00}), (\ref{2.1}) and (\ref{2.5}), writing $p=t-ts$, we have \begin{align}
D^t_s e_i&=(\theta_i^j+(1-p)\gamma^j_i)e_j+s(\ov{\theta_2})^j_i\bar{e}_j=(\theta^p)_i^je_j+s(\ov{\theta_2})^j_i\bar{e}_j. \label{R3}
\end{align}
Denote $\Theta^D$ to be the matrix of curvature 2-forms of $D^t_s$. Then we have
\begin{align}
(\Theta^D)^j_i=d(\theta^p)^j_i-(\theta^p)_i^k\wedge(\theta^p)_k^j-s^2\ov{(\theta_2)}^k_i\wedge(\theta_2)_k^j.
\end{align}
So by (\ref{2.6}), we have \begin{align}R^D_{k\bar{l}i\bar{j}}=R^p_{k\bar{l}i\bar{j}}+s^2\{\ov{T^k_{rj}}T^l_{ir}-(T^j_{\bar{r}\bar{l}}+T^r_{\bar{l}\bar{j}}+T^l_{\bar{r}\bar{j}})(\ov{T^r_{\bar{i}\bar{k}}}+\ov{T^i_{\bar{k}\bar{r}}}
+\ov{T^k_{\bar{i}\bar{r}}})\}.\end{align}
Combing with Lemma \ref{lem2.01}, we have 
\begin{lem}\label{lem2.02}
\begin{align} 
R^D_{k\bar{l}i\bar{j}}&=R_{k\bar{l}i\bar{j}}+p(T^j_{ik,\bar{l}}+\ov{T^i_{jl,\bar{k}}})+(p^2-2p)(T^r_{ik}\ov{T^r_{jl}}-T^j_{rk}\ov{T^i_{rl}})\notag \\
&+(s^2-1)\{\ov{T^k_{rj}}T^l_{ir}-(T^j_{\bar{r}\bar{l}}+T^r_{\bar{l}\bar{j}}+T^l_{\bar{r}\bar{j}})(\ov{T^r_{\bar{i}\bar{k}}}+\ov{T^i_{\bar{k}\bar{r}}}
+\ov{T^k_{\bar{i}\bar{r}}})\}  \label{2.21}
\end{align}
If $J$ is integrable, then  \begin{align*}
R^D_{k\bar{l}i\bar{j}}&=R_{k\bar{l}i\bar{j}}+p(T^j_{ik,\bar{l}}+\ov{T^i_{jl,\bar{k}}})+(p^2-2p)(T^r_{ik}\ov{T^r_{jl}}-T^j_{rk}\ov{T^i_{rl}})+(s^2-1)\ov{T^k_{rj}}T^l_{ir}
\end{align*} 
Here the subscripts in $T$ stands for covariant derivatives with respect to $\nabla^c$. 
\end{lem}
In the rest of the section, we discuss the holomorphic sectional curvature of any metric connection. Given any metric connection $\nabla$, namely $\nabla g=0$, the holomorphic sectional curvature of $\nabla$ on a $J$-invariant tangent plane $\Sigma$ is denoted to be $H^\nabla(\Sigma)$ which is the sectional curvature of $\nabla$ on $\Sigma$. If $\Sigma$ is spanned by $\{X, JX\}$ for $X\in TM\setminus \{0\}$, then $$H^\nabla(\Sigma)=\dfrac{R^\nabla(X,JX,JX,X)}{|X|^4}.$$ Let $\eta=\dfrac{1}{\sqrt{2}}(X-\sqrt{-1}JX)$. It holds that $$H^\nabla(\Sigma)=H^\nabla(\eta)=\dfrac{R^{\nabla}(\eta,\ov{\eta}, \eta,\ov{\eta})}{|\eta|^4}.$$
Assume that $R^{\nabla}_{i\bar{j}k\bar{l}}$ are the components of $R^{\nabla}$. Define its symmetrization tensor components to be 
\begin{align}\label{2.30}
\widehat{R}^{\nabla}_{i\bar{j}k\bar{l}}=\dfrac{1}{4}(R^{\nabla}_{i\bar{j}k\bar{l}}+R^{\nabla}_{k\bar{j}i\bar{l}}+R^{\nabla}_{i\bar{l}k\bar{j}}+R^{\nabla}_{k\bar{l}i\bar{j}}).\end{align}
The following characterization of Hermitian metric of pointwise constant holomorphic sectional curvature is given in \cite{B} (see also \cite{CCN}).
\begin{lem} \label{lem2.03}
At any $q\in M$, $H^{\nabla}(\eta)=c$ for any $\eta\in T_q^{1,0}M\setminus \{0\}$ if and only if $\widehat{R}^{\nabla}_{i\bar{j}k\bar{l}}=\dfrac{c}{2}(\delta_{ij}\delta_{kl}+\delta_{il}\delta_{kj})$ with respect to any unitary frames near $q$.
\end{lem}
 
 \section{self-duality}
In this section, we relate almost Hermitian surfaces with pointwise constant holomorphic sectional curvature to Riemannian self-dual 4-manifolds. We refer to \cite{ADM},\cite{Be} and the reference therein for more details on self-dual 4-manifolds. 

Given a Riemannian 4-manifold $(M,g)$, $g$ induces a metric on the bundle $\Lambda^2TM$. The curvature of the Levi-Civita connection $\nabla^{LC}$ induces a self-adjoint endomorphism $\mathcal R$ of $\Lambda^2TM$ by $g(\mathcal R(X\wedge Y), Z\wedge W)=-R(X,Y,Z,W)$, for $X,Y,Z,W\in TM$. The Hodge $*$ operator on $\Lambda^2TM$ preserves $g$ and satisfies $* ^2=id$. The $+1, -1$ eigenbundles of $*$ are denoted by $\Lambda^2_+, \Lambda^2_-$. Define the Weyl tensor operator to be
$$\mathcal W=\dfrac{1}{2}(\mathcal R+*\mathcal R*)-\dfrac{s_g}{12}id,$$ 
where $s_g$ is the Riemannian scaler curvature of $g$. Define $$\mathcal W_+=\dfrac{1}{2}(\mathcal W+*\mathcal W), \ \ \ \mathcal W_-=\dfrac{1}{2}(\mathcal W-*\mathcal W).$$ 
Then $\mathcal W_+|_{\Lambda^2_+}=\mathcal W|_{\Lambda^2_+}, \mathcal W_+|_{\Lambda^2_-}=0$ and $\mathcal W_-|_{\Lambda^2_-}=\mathcal W|_{\Lambda^2_-}, \mathcal W_-|_{\Lambda^2_+}=0$. 
And for $\al_+, \beta_+\in \Lambda^2_+$, $\al_-, \beta_-\in \Lambda^2_-$, 
\begin{align}
g(\mathcal W_+(\al_+), \beta_+)&=g(\mathcal R(\al_+), \beta_+)-\dfrac{s_g}{12}g(\al_+, \beta_+), \notag\\
g(\mathcal W_-(\al_-), \beta_-)&=g(\mathcal R(\al_-),\beta_-)-\dfrac{s_g}{12}g(\al_-, \beta_-). \label{3.12}
\end{align}
The operators $\mathcal R, \mathcal W, \mathcal W_+, \mathcal W_-$ and the inner product $g$ can be all extended complex linearly to $\Lambda^2TM\otimes \mathbb C$. 
\begin{defn}
$(M,g)$ is called self-dual (anti-self-dual) if $\mathcal W_-=0$ (resp. $\mathcal W_+=0$). 
\end{defn}
For an almost Hermitian surface $(M,J,g)$, fix $q\in M$ and assume $e_1,e_2$ to be a local unitary frame in $T^{1,0}M$. Then
$\{e_1\wedge \bar{e}_2, \dfrac{1}{\sqrt{2}}(e_1\wedge \bar{e}_1-e_2\wedge \bar{e}_2), \bar{e}_1\wedge e_2\}$ forms a unitary basis of $\Lambda^2_-\otimes \mathbb C$. Applying (\ref{3.12}) to this basis, the following characterization of self-duality of almost Hermitian surfaces can be obtained, which has been obtained in Lemma 4.1 in \cite{ADM} for the Hermitian case (see also Lemma 3 in \cite{CZ}) and in the proof of Theorem 1.3 in \cite{LU} for the almost Hermitian case.
\begin{lem} \label{lem3.1}Let $(M,J,g)$ be an almost Hermitian surface. Then $(M,g)$ is self dual if and only if the following equations hold:
\begin{align*}
R_{1\bar{2}1\bar{2}}=0, \ \ \ R_{1\bar{2}2\bar{2}}-R_{1\bar{2}1\bar{1}}=0,\\
2R_{1\bar{2}2\bar{1}}+2R_{1\bar{1}2\bar{2}}-R_{1\bar{1}1\bar{1}}-R_{2\bar{2}2\bar{2}}=0.
\end{align*}
\end{lem}
Then we prove
\begin{prop}\label{prop2.1}
Let $(M,g)$ be an almost Hermitian surface and $(t,s)\in \mathbb R^2$. Assume that the holomorphic sectional curvature with respect to $D^t_s$ is pointwise constant. Then $M$ is self-dual. 
\end{prop}
\begin{proof}
We follow the same steps as in \cite{CZ}. As $D^t_s$ has pointwise constant holomorphic sectional curvature $c$, by Lemma \ref{lem2.03}, for any unitary frame, $$\widehat{R}^{D}_{i\bar{j}k\bar{l}}=\dfrac{c}{2}(\delta_{ij}\delta_{kl}+\delta_{il}\delta_{kj}).$$
In particular, 
\begin{align} \widehat{R}^{D}_{1\bar{2}1\bar{2}}=\widehat{R}^{D}_{1\bar{1}1\bar{2}}=\widehat{R}^{D}_{1\bar{2}2\bar{2}}=0,\notag \\ 
\widehat{R}^{D}_{1\bar{1}1\bar{1}}=\widehat{R}^{D}_{2\bar{2}2\bar{2}}=c, \ \ \ \widehat{R}^{D}_{1\bar{1}2\bar{2}}=\dfrac{c}{2}. \label{3.21}
\end{align}
Next, from (\ref{2.21}) we have the following:
\begin{align*}
\widehat{R}^{D}_{k\bar{l}i\bar{j}}&=\widehat{R}_{k\bar{l}i\bar{j}}-\dfrac{1}{4}(p^2-2p-1+s^2)(T^j_{rk}\ov{T^i_{rl}}+T^j_{ri}\ov{T^k_{rl}}+T^l_{rk}\ov{T^i_{rj}}+T^l_{ri}\ov{T^k_{rj}})\\
&+(1-s^2)(T^j_{\bar{r}\bar{l}}+T^l_{\bar{r}\bar{j}})(\ov{T^i_{\bar{k}\bar{r}}}
+\ov{T^k_{\bar{i}\bar{r}}}). 
\end{align*}
In particular, letting $b=p^2-2p-1+s^2$, we get 
\begin{align}
 \widehat{R}^{D}_{1\bar{2}1\bar{2}}&= \widehat{R}_{1\bar{2}1\bar{2}}, \notag\\
 \widehat{R}^{D}_{1\bar{2}2\bar{2}}&=\widehat{R}_{1\bar{2}2\bar{2}}-\frac{1}{2}bT^2_{12}\ov{T^1_{12}}+2(1-s^2)T^2_{\bar{1}\bar{2}}\ov{T^1_{\bar{2}\bar{1}}}, \notag\\
 \widehat{R}^{D}_{1\bar{1}1\bar{2}}&=\widehat{R}_{1\bar{1}1\bar{2}}-\frac{1}{2}bT^2_{21}\ov{T^1_{21}}+2(1-s^2)T^2_{\bar{2}\bar{1}}\ov{T^1_{\bar{1}\bar{2}}},\notag\\
\widehat{R}^{D}_{1\bar{1}1\bar{1}}&=\widehat{R}_{1\bar{1}1\bar{1}}-b|T^1_{21}|^2-4(1-s^2)|T^1_{\bar{2}\bar{1}}|^2,\notag\\
\widehat{R}^{D}_{2\bar{2}2\bar{2}}&=\widehat{R}_{2\bar{2}2\bar{2}}-b|T^2_{12}|^2-4(1-s^2)|T^2_{\bar{1}\bar{2}}|^2,\notag\\
\widehat{R}^{D}_{1\bar{1}2\bar{2}}&=\widehat{R}_{1\bar{1}2\bar{2}}-\frac{1}{4}b(|T^2_{21}|^2+|T^1_{12}|^2)-(1-s^2)(|T^2_{\bar{1}\bar{2}}|^2+|T^1_{\bar{1}\bar{2}}|^2).\notag
 \end{align}
 So from (\ref{3.21}), we get
 \begin{align}
  \widehat{R}_{1\bar{2}1\bar{2}}=0,\ \ \  \widehat{R}_{1\bar{1}1\bar{2}}-\widehat{R}_{1\bar{2}2\bar{2}}=0 \notag\\
  4\widehat{R}_{1\bar{1}2\bar{2}}-\widehat{R}_{1\bar{1}1\bar{1}}-\widehat{R}_{2\bar{2}2\bar{2}}=0 \label{3.22}
  \end{align}
 For curvatures of the Levi-Civita connection,  we have
 \begin{align*}
 \widehat{R}_{1\bar{2}1\bar{2}}&={R}_{1\bar{2}1\bar{2}},\ \ \ \widehat{R}_{1\bar{1}1\bar{2}}=\dfrac{1}{2}({R}_{1\bar{1}1\bar{2}}+{R}_{1\bar{2}1\bar{1}})={R}_{1\bar{2}1\bar{1}},\\
 \widehat{R}_{1\bar{2}2\bar{2}}&=\dfrac{1}{2}({R}_{1\bar{2}2\bar{2}}+{R}_{2\bar{2}1\bar{2}})={R}_{1\bar{2}2\bar{2}},\\
 \widehat{R}_{1\bar{1}2\bar{2}}&=\dfrac{1}{4}({R}_{1\bar{1}2\bar{2}}+{R}_{2\bar{1}1\bar{2}}+{R}_{1\bar{2}2\bar{1}}+{R}_{2\bar{2}1\bar{1}})=\dfrac{1}{2}({R}_{1\bar{1}2\bar{2}}+{R}_{1\bar{2}2\bar{1}}),\\
 \widehat{R}_{1\bar{1}1\bar{1}}&={R}_{1\bar{1}1\bar{1}}, \ \ \ \widehat{R}_{2\bar{2}2\bar{2}}={R}_{2\bar{2}2\bar{2}}
 \end{align*}
 Therefore, putting them into (\ref{3.22}), we get
 \begin{align*}
  {R}_{1\bar{2}1\bar{2}}=0,\ \ \ \ {R}_{1\bar{2}1\bar{1}}-{R}_{1\bar{2}2\bar{2}}=0 \\
  2({R}_{1\bar{1}2\bar{2}}+{R}_{1\bar{2}2\bar{1}})-{R}_{1\bar{1}1\bar{1}}-R_{2\bar{2}2\bar{2}}=0
  \end{align*}
  These are exactly the equations in Lemma \ref{lem3.1}. So we finish the proof.
\end{proof}
\begin{remk}
When $(M,g)$ is almost K\"ahler, i.e. $d\omega=0$, the above result for the Gauduchon connection $\nabla^t$ has been proved in \cite{LU}.
\end{remk}

\section{Conformal changes for curvatures of canonical connections}
In this section, we compute the transformation of curvature components of the canonical connections $\nabla^t$ and $D^t_s$ under conformal change of Hermitian metrics on a complex manifold. This will be applied in the next sections to prove the main results. 

We first compute the case for $\nabla^t$ which is more illustrative. Let $g$ be a Hermitian metric on a complex manifold $(M,J)$. Fix any $q\in M$. As before, let $\{e_1, e_2, ..., e_n\} $ be a local unitary frame of $(1,0)$ vectors near $q$ with $\{\vp^1, \vp^2, ..., \vp^n\}$ being the dual frame. 
For any $t\in \mathbb R$, denote $\nabla^t$ to be the $t$-Gauduchon connection of $g$. Write $\nabla^te_i=(\theta^t)^j_ie_j$. Then
\begin{align}\label{c00}
(\theta^t)_i^j=\theta_i^j+(1-t)(T_{ik}^j\vp^k-\ov{T_{jk}^i}\ov{\vp}^k),
\end{align}
where $\theta$ is the connection matrix of the Chern connection $\nabla^c$ of $g$ and $\{T_{ik}^j\}$ are the components of the torsion forms of the $\nabla^c$ by $\tau^i=T^i_{jk}\varphi^j\wedge\varphi^k$ with $T^i_{jk}=-T^i_{kj}$. Note that there are no components of $T^i_{\bar{j}\bar{k}}$  since $J$ is integrable. 

Let $\wtl{g}=e^{2f}g$ be a metric conformal to $g$. Define $\wtl{e}_i=e^{-f}e_i$ and $\wtl{\vp}^i=e^f\vp^i$. Then $\{\wtl{e}_i\}$ forms a unitary frame of $\wtl{g}$ near $q$ with the unitary coframe being $\{\wtl{\vp}^i\}$. Denote $\wtl{\nabla}^c$ to be the Chern connection of $\wtl{g}$ with $\wtl{\theta}$ being the connection matrix and $\{\wtl{T}_{jk}^i\}$ being the components of the torsion forms $\wtl{\tau}^i$ of $\wtl{\nabla}^c$. Direct calculations from the structure equations (see \cite{CZ}, \cite{YZZ}) give that
\begin{align}
(\wtl{\theta})_i^j=\theta_i^j+(\pa f-\bar{\pa}f)\delta_{ij}, \ \ \ \  \tl{\tau}^i=e^f(\tau^i+2\pa f\vp^i). \label{4.01}
\end{align}
Then \begin{align}
    \wtl{T}_{jk}^i=e^{-f}(T_{jk}^i+f_j\de_{ik}-f_k\de_{ij}),
\end{align}
where $f_j=e_jf$. Denote $\wtl{\nabla}^t$ to be the $t$-Gauduchon connection of $\wtl{g}$. Assume that $\wtl{\nabla}^t\wtl{e}_i=(\wtl{\theta}^t)^j_i\wtl{e}_j$. 
Then by (\ref{c00}), it follows that
\begin{align}
( \wtl{\theta}^t)_i^j-(\theta^t)_i^j=t(\pa f-\bar{\pa}f)\delta_{ij}+(1-t)(f_i\vp^j-f_{\bj}\ov{\vp^i}). \label{4.22}
\end{align}
 Denote $\Theta^t$ and $\wtl{\Theta}^t$ to be the curvature matrix of $\nabla^t$ and $\wtl{\nabla}^t$. As \begin{align*}
(\Theta^t)_i^j=d(\theta^t)_i^j-(\theta^t)_i^k\wedge (\theta^t)_k^j, \ \ \ \ (\wtl{\Theta}^t)_i^j=d(\wtl{\theta}^t)_i^j-(\wtl{\theta}^t)_i^k\wedge (\wtl{\theta}^t)_k^j,
\end{align*}
by (\ref{4.22}), we have 
\begin{align}&[(\wtl{\Theta}^t)_i^j]^{1,1}-[(\Theta^t)_i^j]^{1,1} \label{4.23}\\
 =&-2t\pa\bar{\pa}f\delta_{ij}+(1-t)(\bar{\pa}(f_i\vp^j)-\pa(f_{\bj}\ov{\vp^i}))-(1-t)^2(f_kf_{\bar{k}}\vp^j\wg \ov{\vp^i}-f_if_{\bj}\vp^k\wg\ov{\vp^k}). \notag
\end{align}
Here $[(\wtl{\Theta}^t)_i^j]^{1,1}$ means the $(1,1)$ part of the curvature forms.
Let \begin{align*}
R^t_{k\bar{l}i\bar{j}}&=(\Theta^t)_i^j(e_k, e_{\bar{l}})=R^t(e_k, e_{\bar{l}}, e_i, e_{\bar{j}}),\\
\wtl{R}^t_{k\bar{l}i\bar{j}}&=(\wtl{\Theta}^t)_i^j(\wtl{e}_k, \wtl{e}_{\bar{l}})=\wtl{R}^t(\wtl{e}_k, \wtl{e}_{\bar{l}}, \wtl{e}_i, \wtl{e}_{\bar{j}}).
\end{align*}
By Lemma 1 in \cite{CZ}, we may choose a local unitary frame of (1,0) vectors near $q$ such that $\theta^t|q=0$. 
By (\ref{c00}) and the structure equation $d\vp^i=-\theta^i_j\wg \vp^j+\tau^i$, we have at $q$
\begin{align*}
\bar{\pa}\vp^i=-(\theta^i_j)^{0,1}\wg \vp^j=(1-t)\ov{T^k_{il}}\vp^k\wg \ov{\vp^l}.
\end{align*}
Let $f_{k\bl}=\ov{e_l}e_kf-(\na^t_{\ov{e_l}}e_k)f,\ f_{\bl k}=e_k\ov{e_l}f-(\na^t_{e_k}\ov{e_l})f$, i.e., $f_{k\bl}$ and $f_{\bl k}$ are the covariant derivatives of $f$ with respect to $\nabla^t$. Then 
\begin{align}\label{comm}
f_{\bj k}-f_{k\bj}=(1-t)(f_{\bar{r}}T_{rk}^j-f_r\ov{T_{rj}^k}).
\end{align}
Then at $q$, 
\begin{align}
\pa\bar{\pa}f(e_k,\ov{e_l})=-\bar{\pa}\pa f(e_k,\ov{e_l})=f_{k\bl}-(1-t)f_r\ov{T_{rl}^k}. \label{4.03}
\end{align} Also, 
\begin{align}
&[\bar{\pa}(f_i\vp^j)-\pa(f_{\bj}\ov{\vp^i})](e_k,\ov{e_l})\notag\\
=&-f_{i\bl}\de_{jk}-f_{k\bj}\de_{il}+(1-t)(f_i\ov{T_{jl}^k}-f_{\bar{r}}T_{rk}^j\de_{il}+f_r\ov{T_{rj}^k}\de_{il}+f_{\bj}T_{ik}^l).\label{4.04}
\end{align}
It follows from (\ref{4.23}) that 
\begin{align} 
&e^{2f}\wtl{R}^t_{k\bar{l}i\bar{j}}-R^t_{k\bar{l}i\bar{j}} \notag \\
=&-2tf_{k\bl}\delta_{ij}+2t(1-t)f_r\ov{T_{rl}^k}\delta_{ij}-(1-t)(f_{i\bl}\delta_{jk}+f_{k\bj}\de_{il}) \notag\\
&+(1-t)^2(f_i\ov{T_{jl}^k}-f_{\bar{r}}T_{rk}^j\de_{il}+f_r\ov{T_{rj}^k}\de_{il}+f_{\bj}T_{ik}^l-f_rf_{\bar{r}}\delta_{jk}\delta_{il}+f_if_{\bj}\delta_{kl})\label{h000}
\end{align}
Note that the above equality is tensorial and then holds for any unitary frames.\\
Denote $\widehat{R}^t_{i\bar{j}k\bar{l}}$ and 
$\widehat{R}^t_{i\bar{j}k\bar{l}}$ to be the symmetrization of the curvature $R^t$ and $\wtl{R}^t$ as defined in (\ref{2.30}). From (\ref{h000}) we get
\begin{align*}
&e^{2f}\widehat{\wtl{R}^t}_{k\bar{l}i\bar{j}}-\hat{R^t}_{k\bar{l}i\bar{j}}\\
=&-\dfrac{1}{2}(f_{i\bl}\delta_{jk}+f_{k\bl}\delta_{ij}+f_{i\bj}\delta_{kl}+f_{k\bj}\delta_{il})+\dfrac{(1-t)^2}{4}(f_if_{\bj}\delta_{kl}+f_kf_{\bj}\delta_{il}+f_if_{\bl}\delta_{kj}+f_kf_{\bl}\delta_{ij})\\
&-\dfrac{(1-t)^2}{4}(f_{\bar{r}}T_{rk}^j\delta_{il}+f_{\bar{r}}T_{rk}^l\delta_{ij}+f_{\bar{r}}T_{ri}^j\delta_{kl}+f_{\bar{r}}T_{ri}^l\delta_{kj}+2f_rf_{\bar{r}}(\delta_{kl}\delta_{ij}+\delta_{jk}\delta_{il}))\\
&+\dfrac{1-t^2}{4}(f_r\ov{T_{rl}^k}\delta_{ij}+f_r\ov{T_{rj}^k}\delta_{il}+f_r\ov{T_{rl}^i}\delta_{kj}+f_r\ov{T_{rj}^i}\delta_{kl}).
\end{align*}
Assume that $g$ is K\"ahler. Then $R^t_{i\bar{j}k\bar{l}}=R_{i\bar{j}k\bar{l}}=\widehat{R}_{i\bar{j}k\bar{l}}$ for any $t\in \mathbb R$ and $T^i_{jk}=0$. So
\begin{align}  
e^{2f}\widehat{\wtl{R}^t}_{k\bar{l}i\bar{j}}-R_{k\bar{l}i\bar{j}}=&-\dfrac{1}{2}(f_{i\bl}\delta_{jk}+f_{k\bl}\delta_{ij}+f_{i\bj}\delta_{kl}+f_{k\bj}\delta_{il}) \notag \\
&+\dfrac{(1-t)^2}{4}(f_if_{\bj}\delta_{kl}+f_kf_{\bj}\delta_{il}+f_if_{\bl}\delta_{kj}+f_kf_{\bl}\delta_{ij}).\notag \\
&- \dfrac{(1-t)^2}{2}f_rf_{\bar{r}}(\delta_{kj}\de_{il}+\de_{ij}\de_{kl}).\label{h11}
\end{align}

Next, consider the two-parameter canonical connection $D^t_s=(1-s)\nabla^t+s\nabla^{LC}$. Then by (\ref{R3}), we have 
$$D^t_s e_i=(\theta^p)_i^je_j+s(\ov{\theta_2})^j_i\bar{e}_j,$$
where $p=t-ts$ and $(\theta_2)^j_i=\ov{T^k_{ij}}\vp^k$ as $J$ is integrable. Denote $\Theta^D$ to be the $(n\times n)$ curvature matrix of 2-forms of $D^t_s$ given by $(\Theta^D)^j_i(X,Y)=R^D(X,Y,e_i,\bar{e}_j)$ for $X,Y\in TM$. 
Then by the structure equation, 
\begin{align}
(\Theta^D)_i^j&=d(\theta^p)_i^j-(\theta^p)_i^k\wedge (\theta^p)_k^j-s^2(\ov{\theta_2})_i^k\wedge(\theta_2)_k^j \notag\\
&=d(\theta^p)_i^j-(\theta^p)_i^k\wedge (\theta^p)_k^j-s^2T_{ik}^l\ov{T_{kj}^r}\ov{\vp^l}\wedge \vp^r. \label{4.24}\end{align}
Denote $\wtl{D}^t_s$ to be the $(r,s)$ canonical connection of $\wtl{g}$ with $\wtl{\theta}^D, \wtl{\Theta}^D$ being the connection and curvature matrix of $\wtl{D}^t_s$. Similarly,
\begin{align}
(\wtl{\Theta}^D)_i^j&=d(\wtl{\theta}^p)_i^j-(\wtl{\theta}^p)_i^k\wedge (\wtl{\theta}^p)_k^j-s^2\wtl{T}_{ik}^l\ov{\wtl{T}_{kj}^r}\ov{\wtl{\vp}^l}\wedge \wtl{\vp}^r. \label{4.25}\end{align}
By (\ref{4.22}), (\ref{4.24}) and (\ref{4.25}), we get
\begin{align*}
&[(\wtl{\Theta}^D)^j_i]^{1,1}-[(\Theta^D)_i^j]^{1,1}\\
=&-2p\pa\bar{\pa}f\de_{ij}+(1-p)(\bar{\pa}(f_i\vp^j)-\pa(f_j\ov{\vp}^i))\\
&-((p-1)^2+s^2)f_rf_{\bar{r}}\vp^j\wg\ov{\vp}^i+((p-1)^2-s^2)f_if_{\bj}\vp^r\wg\ov{\vp}^r\\
&-s^2(f_{\bj}T^l_{ik}\vp^k\wg\ov{\vp^l}-f_i\ov{T^k_{lj}}\vp^k\wg\ov{\vp^l}+f_{\bar{r}}T^l_{ri}\vp^j\wg\ov{\vp^l}+f_r\ov{T^k_{rj}}\vp^k\wg\ov{\vp^i})\\
&+s^2(f_i\vp^j\wg\bar{\pa}f+f_{\bj}\pa f\wg \ov{\vp}^i)
\end{align*}
Choose a local unitary frame of (1,0) vectors near $q$ such that $\theta^p|q=0$. By (\ref{c00}) and the structure equation, we have
\begin{align*}
\bar{\pa}\vp^i=-(\theta^i_j)^{0,1}\wg \vp^j=-(1-p)\ov{T^j_{il}}\ov{\vp^l}\wg\vp^j.
\end{align*}
Then at $q$, by (\ref{4.23}) and (\ref{h000}), we have
\begin{align}
&e^{2f}\wtl{R}^D_{k\bar{l}i\bar{j}}-R^D_{k\bar{l}i\bar{j}}\notag \\
=&-2pf_{k\bl}\delta_{ij}+2p(1-p)f_r\ov{T_{rl}^k}\delta_{ij}-(1-p)(f_{i\bl}\delta_{jk}+f_{k\bj}\de_{il}) \notag\\
&+(1-p)^2(f_i\ov{T_{jl}^k}-f_{\bar{r}}T_{rk}^j\de_{il}+f_r\ov{T_{rj}^k}\de_{il}+f_{\bj}T_{ik}^l-f_rf_{\bar{r}}\delta_{jk}\delta_{il}+f_if_{\bj}\delta_{kl})\notag \\
&+s^2(f_if_{\bl}\de_{jk}+f_{\bj}f_k\de_{il}-f_if_{\bj}\delta_{kl}-f_rf_{\bar{r}}\delta_{jk}\delta_{il})\notag \\
&+s^2(f_i\ov{T^k_{lj}}-f_{\bj}T^l_{ik}-f_{\bar{r}}T^l_{ri}\de_{jk}-f_r\ov{T^k_{rj}}\de_{il}).\label{4.33}
\end{align}
Here, the subscripts under $f$ stand for covariant derivatives with respect to $\nabla^p$. As before, the above equality holds under any unitary frames since it is tensorial.
Denote $\widehat{\wtl{R}^D}_{k\bar{l}i\bar{j}}$ and $\widehat{R^D_{k\bar{l}i\bar{j}}}$ to be the components of symmetrization of $\wtl{R}^D$ and $R^D$. Then 
\begin{align*}
&e^{2f}\widehat{\wtl{R}^D}_{k\bar{l}i\bar{j}}-\widehat{R^D_{k\bar{l}i\bar{j}}}\\
=&-\dfrac{1}{2}(f_{i\bl}\delta_{jk}+f_{k\bl}\delta_{ij}+f_{i\bj}\delta_{kl}+f_{k\bj}\delta_{il})+\dfrac{(p-1)^2+s^2}{4}(f_if_{\bj}\delta_{kl}+f_kf_{\bj}\delta_{il}+f_if_{\bl}\delta_{kj}+f_kf_{\bl}\delta_{ij})\\
&-\dfrac{(p-1)^2+s^2}{4}(f_{\bar{r}}T_{rk}^j\delta_{il}+f_{\bar{r}}T_{rk}^l\delta_{ij}+f_{\bar{r}}T_{ri}^j\delta_{kl}+f_{\bar{r}}T_{ri}^l\delta_{kj}+2f_rf_{\bar{r}}(\delta_{kl}\delta_{ij}+\delta_{jk}\delta_{il}))\\
&+\dfrac{1-p^2-s^2}{4}(f_r\ov{T_{rl}^k}\delta_{ij}+f_r\ov{T_{rj}^k}\delta_{il}+f_r\ov{T_{rl}^i}\delta_{kj}+f_r\ov{T_{rj}^i}\delta_{kl}).
\end{align*}
When $g$ is K\"ahler, we get
\begin{align}e^{2f}\widehat{\wtl{R}^D}_{k\bar{l}i\bar{j}}-\widehat{R^D_{k\bar{l}i\bar{j}}}=&-\dfrac{1}{2}(f_{i\bl}\delta_{jk}+f_{k\bl}\delta_{ij}+f_{i\bj}\delta_{kl}+f_{k\bj}\delta_{il})\notag\\
&+\dfrac{(p-1)^2+s^2}{4}(f_if_{\bj}\delta_{kl}+f_kf_{\bj}\delta_{il}+f_if_{\bl}\delta_{kj}+f_kf_{\bl}\delta_{ij})\notag\\
&-\dfrac{(p-1)^2+s^2}{2}f_rf_{\bar{r}}(\delta_{kl}\delta_{ij}+\delta_{jk}\delta_{il}))\label{h22}
\end{align}

\section{Isosceles Hopf manifolds}
In this section, we apply the formulae in section 4 to study Hermitian metrics which are conformal to the standard Hopf metric and have pointwise constant holomorphic sectional curvature with respect to $\nabla^t$ or $D^t_s$ on isosceles Hopf manifolds. Recall that an isosceles Hopf manifold $M_{\sigma}$ (\cite{CZ}) is the quotient of $\mathbb C^n\setminus \{0\}$ by the infinite cyclic group $\mathbb Z$ generated by $$\sigma: (z_1,\cdots, z_n)\longrightarrow (a_1z_1,\cdots, a_nz_n)$$
where $0<a=|a_1|=\cdots=|a_n|<1$. Denote $\omega_0=\sqrt{-1}(dz_1\wedge d\bar{z}_1+\cdots+dz_n\wedge d\bar{z}_n)$ to be the K\"ahler form of the standard Euclidean metric $g_0$ on $\mathbb C^n$ and $\omega_h=\dfrac{\sqrt{-1}}{|z|^2}(dz_1\wedge d\bar{z}_1+\cdots+dz_n\wedge d\bar{z}_n)$, where $|z|^2={}^{t}\bar{z}z=|z_1|^2+\cdots+|z_n|^2$ with $^tz=(z_1,\cdots, z_n)$ ($z$ is viewed as a column vector). As $\omega_h$ is invariant under $\sigma$, it induces a Hermitian metric on $M_{\sigma}$ which is called the \textbf{standard} Hopf metric. 

In \cite{CZ}, Chen-Zheng study Hermitian metrics which are conformal to $\omega_h$ and have pointwise constant Strominger holomorphic sectional curvature and prove the following result (Proposition 1 there).
\begin{prop}[\cite{CZ}] \label{prop4.1}
Let $M_{\sigma}$ be an isosceles Hopf manifold. A Hermitian metric $\wtl{\omega}$ on $M_{\sigma}$ conformal to $\omega_h$ has pointwise constant Strominger holomorphic sectional curvature if and only if  $\wtl{\omega}=\dfrac{c_0}{|z|^2+ ^tzAz+\overline{^tzAz}}\omega_0$ for some constant $c_0>0$ and symmetric $(n\times n)$ complex matrix $A$ satisfying $A\bar{A}<\dfrac{1}{4}I_n$ and $D_\sigma AD_\sigma=a^2A$, where $D_\sigma=diag\{a_1,\cdots, a_n\}$. $\wtl{\omega}$ is then called an \textbf{admissible} metric. \end{prop}
In the following, we will generalize Proposition \ref{prop4.1} to the case for $\nabla^t$ and $D^t_s$. Assume that $\wtl{g}$ is a Hermitian metric which is conformal to the standard Hopf metric. Write $\wtl{\omega}$ for the K\"ahler form of $\wtl{g}$ such that $\wtl{\omega}=F\omega_h=e^{2f}\omega_0$, where $e^{2f}=\dfrac{F}{|z|^2}$ and $F$ is a positive function on $M_{\sigma}$. Take $\mathbb C^n\setminus \{0\}$ as a global parameter space of $M_{\sigma}$ and view $F$ and $f$ as functions on $\mathbb C^n\setminus \{0\}$. Then $\wtl{g}$ is conformal to the flat metric $g_0$ with the conformal factor $e^{2f}$. As $\sigma^*|z|^2=a^2|z|^2$ and $\sigma^*F=F$, we get $\sigma^*(e^{2f})=a^{-2}e^{2f}$. Now in the notation of section 4, as $g=g_0$ is flat, $R^t=0$ and $T=0$. Let $\wtl{R}^t$ and $\widehat{\wtl{R}^t}$ be the curvature and its symmetrization of the $t$-Gauduchon connection of $\wtl{g}$. Choose $e_i=\frac{\pa}{\pa z_i}$, $1\leq i\leq n$. By (\ref{h11}), we have
\begin{align}  
e^{2f}\widehat{\wtl{R}^t}_{k\bar{l}i\bar{j}}=&-\dfrac{1}{2}(f_{i\bl}\delta_{jk}+f_{k\bl}\delta_{ij}+f_{i\bj}\delta_{kl}+f_{k\bj}\delta_{il}) \notag \\
&+\dfrac{(1-t)^2}{4}(f_if_{\bj}\delta_{kl}+f_kf_{\bj}\delta_{il}+f_if_{\bl}\delta_{kj}+f_kf_{\bl}\delta_{ij}).\notag \\
&- \dfrac{(1-t)^2}{2}f_rf_{\bar{r}}(\delta_{kj}\de_{il}+\de_{ij}\de_{kl}),\label{h111}
\end{align}
where $f_r, f_{i\bar{j}}$ are now ordinary partial derivatives. Assume that $\wtl{g}$ has pointwise constant holomorphic sectional curvature $\wtl{c}$ with respect to $\wtl{\nabla}^t$, for some $t\in \mathbb R$. By Lemma \ref{lem2.03}, \begin{align}\label{5.1}
\widehat{\wtl{R}^t}_{k\bar{l}i\bar{j}}=\dfrac{\wtl{c}}{2}(\delta_{ij}\delta_{kl}+\delta_{il}\delta_{kj}).\end{align}
Putting (\ref{5.1}) into (\ref{h111}), we have
\begin{align}  
(\dfrac{\wtl{c}}{2}e^{2f}+\dfrac{(1-t)^2}{2}f_rf_{\bar{r}})(\delta_{ij}\delta_{kl}+\delta_{il}\delta_{kj})=-\dfrac{1}{2}(f_{i\bl}\delta_{jk}+f_{k\bl}\delta_{ij}+f_{i\bj}\delta_{kl}+f_{k\bj}\delta_{il}) \notag \\
+\dfrac{(1-t)^2}{4}(f_if_{\bj}\delta_{kl}+f_kf_{\bj}\delta_{il}+f_if_{\bl}\delta_{kj}+f_kf_{\bl}\delta_{ij}).\label{h112} 
\end{align}
We first argue that $t\neq 1$. Otherwise, if $t=1$, then (\ref{h112}) becomes 
\begin{align}  
-\wtl{c}e^{2f}(\delta_{ij}\delta_{kl}+\delta_{il}\delta_{kj})=f_{i\bl}\delta_{jk}+f_{k\bl}\delta_{ij}+f_{i\bj}\delta_{kl}+f_{k\bj}\delta_{il}.\label{h113} 
\end{align}
When $n\geq 2$, by testing different indices in (\ref{h113}), we get 
\begin{align}
f_{i\bj}=-\dfrac{\wtl{c}e^{2f}}{2}\de_{\ij}.\label{h114}
\end{align}
For $i\neq j$, $f_{i\bi\bj}=f_{i\bj\bi}=0$. So $\frac{\pa }{\pa \bar{z}_j}(\wtl{c}e^{2f})=0$. As $\wtl{c}e^{2f}$ is a real function, we get that $\wtl{c}e^{2f}$ is a constant. But this contradicts with $\sigma^*\wtl{c}=\wtl{c}$ and $\sigma^*e^{2f}=a^{-2}e^{2f}$. So $t\neq 1$ and $\mu=-\dfrac{(1-t)^2}{2}\neq 0$. Let $\xi=e^{\mu f}$. Then
\begin{align}\xi_i=\mu\xi f_i, \ \ \ \ \xi_{i\bj}=\mu \xi(f_{i\bj}+\mu f_if_{\bj}). \label{h33}
\end{align}
Putting (\ref{h33}) into (\ref{h112}), we get
\begin{align}  
(-\mu\xi e^{2f}\wtl{c}+\dfrac{2}{\xi}\xi_r\xi_{\bar{r}})(\delta_{kj}\de_{il}+\de_{ij}\de_{kl})=\xi_{i\bl}\delta_{jk}+\xi_{k\bl}\delta_{ij}+\xi_{i\bj}\delta_{kl}+\xi_{k\bj}\delta_{il}. \label{5.2}
\end{align}
Denote $\la=-\dfrac{1}{2}(\mu\xi e^{2f}\wtl{c}-\dfrac{2}{\xi}\xi_r\xi_{\bar{r}})$. By testing different indices in (\ref{5.2}), we get 
\begin{align}
\xi_{i\bj}=\la\de_{\ij}.\label{5.3}
\end{align}
Then for $i\neq j$, $\la_{\bj}=\xi_{i\bi \bj}=\xi_{i\bj\bi}=0$. So $\la$ is a constant. From (\ref{5.3}), we get that $\pa\bar{\pa}(\xi-\la|z|^2)=0$. Then we may write $$\xi=\la|z|^2+\vp+\ov{\vp}$$ for some holomorphic function $\vp$ on $\mathbb C^n\setminus\{0\}$. As $\sigma^*(e^{2f})=a^{-2}e^{2f}$ and $\xi=e^{\mu f}$, we have $\sigma^*\xi=a^{-\mu}\xi$. It follows that
\begin{align*}
\sigma^*(\vp+\ov{\vp})-a^{-\mu}(\vp+\ov{\vp})=\la (a^{-\mu}-a^2)|z|^2.
\end{align*}
As $\vp$ is holomorphic, $\pa\bar{\pa}(\la(a^{-\mu}-a^2)|z|^2)=0$.  Therefore 
\begin{align}
\la (a^{-\mu}-a^2)=0. \label{5.01}
\end{align}
We then divide the discussion into two cases.  \\
(i) If $t\neq -1$ and $t\neq 3$, then $\mu=-\dfrac{(1-t)^2}{2}\neq -2$. By (\ref{5.01}), we have $\la=0$ .
Thus $\vp+\ov{\vp}=\xi=e^{\mu f}>0$, i.e., $\vp$ has positive real part on $\mathbb C^n\setminus\{0\}$. When $n\geq 2$,  by Hartogs extension, $\vp$  is holomorphic on $\mathbb C^n$. As $\sigma^*\vp=a^{-\mu}\vp$. Expanding $\vp$ in Taylor series, we see that $\vp$ is a homogeneous polynomial in $z$ with degree $-\mu$. This is impossible since $\vp$ has positive real part on $\mathbb C^n\setminus\{0\}$. \\
(ii) $t=-1$ or $t=3$, then $\mu=-2$. In this case, $\xi=e^{-2f}$ and (\ref{h111}) becomes
\begin{align}  
\widehat{\wtl{R}^t}_{k\bar{l}i\bar{j}}=&\dfrac{1}{4}(\xi_{i\bl}\delta_{jk}+\xi_{k\bl}\delta_{ij}+\xi_{i\bj}\delta_{kl}+\xi_{k\bj}\delta_{il})-\dfrac{1}{2\xi}\xi_r\xi_{\bar{r}}(\delta_{kj}\de_{il}+\de_{ij}\de_{kl}).\label{h231}
\end{align}
If $\xi$ satisfies $\xi_{i\bj}=\la\de_{\ij}$ for some positive constant $\la$ and $\sigma^*\xi=a^2\xi$, then by (\ref{h231}), $\wtl{g}=\dfrac{1}{\xi}g_0$ would define a Hermitian metric on $M_{\sigma}$ and has pointwise constant holomorphic sectional curvature $\la-\dfrac{1}{\xi}\xi_r\xi_{\bar{r}}$. This is exactly in the same situation of the Strominger connection case (see Page 8 in \cite{CZ}), which gives that $\xi=\la(|z|^2+{}^{t}zAz+\overline{{}^{t}zAz})$. Here $A$ is a symmetric $(n\times n)$ complex matrix satisfying $A\bar{A}<\dfrac{1}{4}I_n$ and $D_\sigma AD_\sigma=a^2A$ with $D_\sigma=diag\{a_1,\cdots, a_n\}$. In other words, $\dfrac{1}{\xi}g_0$ is an admissible metric on $M_{\sigma}$. 
In summary, we have proved the following result which generalizes Proposition 5.1.
\begin{prop} \label{prop4.2}
Let $M_{\sigma}$ be an isosceles Hopf manifold. A Hermitian metric $\wtl{\omega}$ on $M_{\sigma}$ conformal to $\omega_h$ has pointwise constant holomorphic sectional curvature with respect to $\wtl{\nabla}^t$ if and only if $t=-1$ or $t=3$ and $\wtl{\omega}$ is an admissible metric. \end{prop}
Note that from the above calculations, for an admissible metric $\dfrac{1}{\xi_A}\omega_0$ with $\xi_A=|z|^2+{}^{t}zAz+\overline{{}^{t}zAz}$, the holomorphic sectional curvatures with respect to $\wtl{\nabla}^t$ for $t=-1$ and $t=3$ are both $-\dfrac{1}{\xi_A}(4\ {}^{t}zA\bar{A}\bar{z}+{}^{t}zAz+\overline{^tzAz})$. If $A\neq 0$, it is impossible to be a constant (c.f. \cite{CZ}).

Using the same techniques, we also derive the following more general result for the two-parameter connection $\wtl{D}^t_s$ whose proof is omitted.
\begin{prop} \label{prop4.3}
A Hermitian metric $\wtl{\omega}$ on $M_{\sigma}$ which is conformal to $\omega_h$ has pointwise constant holomorphic sectional curvature with respect to $\wtl{D}^t_s$ if and only if $(1-t+ts)^2+s^2=4$ and $\wtl{\omega}$ is an admissible metric. 
In this case, the holomorphic sectional curvatures of $\wtl{D}^t_s$ is $-\dfrac{1}{\xi_A}(4\ {}^{t}zA\bar{A}\bar{z}+{}^{t}zAz+\overline{{}^{t}zAz})$ up to a scalar. \end{prop}
When $n=2$, $M_{\sigma}$ is an isosceles Hopf surface. Let $A=0$ and $\xi_A=|z|^2$. Then $\wtl{\omega}=\dfrac{1}{\xi_A}\omega_0=\omega_h$. So the standard Hopf metric has zero holomorphic sectional curvature for $\wtl{\nabla}^s$ and $\wtl{\nabla}^3$ (or generally $\wtl{D}^t_s$ with $(1-t+ts)^2+s^2=4$). Putting $e^{-2f}=|z|^2$ into (\ref{h000}), direct calculation gives that the curvature for $\wtl{\nabla}^s$ is zero but nonzero for $\wtl{\nabla}^3$. For example, when $t=3$, we have $\wtl{R}^t_{k\bar{l}i\bar{j}}=3(\delta_{kl}\delta_{ij}-\delta_{il}\delta_{jk})+\dfrac{1}{|z|^2} (\bar{z}_kz_j\delta_{il}+\bar{z}_iz_l\delta_{jk}+\bar{z}_iz_j\delta_{kl}-3\bar{z}_kz_l\delta_{ij})$. Then $\wtl{R}^t_{1\bar{1}2\bar{2}}\neq 0$ for $t=3$. Therefore, we have
\begin{cor} \label{cor5.4}
There exist non-flat compact Hermitian surfaces with zero Gauduchon holomorphic sectional curvature.
\end{cor}
Recall that for a K\"ahler manifold, the holomorphic sectional curvature determines the whole curvature tensor. In particular, a K\"ahler metric with zero holomorphic sectional curvature must be flat. The above example shows that it does not hold for compact Hermitian manifolds with respect to general Gauduchon connections even in dimension two. 

\section{Proof of the theorems}
In this section, we prove Theorem \ref{thm0}, Theorem \ref{thm2} and Corollary \ref{cor1.8}. As Theorem \ref{thm2} is a generalization of Theorem \ref{thm0}, it suffices to prove Theorem \ref{thm2}. Using the notations in section 4, assume that $(M,\wtl{g})$ is a compact Hermitian surface with pointwise constant holomorphic sectional curvature with respect to $\wtl{D}^t_s$ for some $(t,s)\in \mathbb R^2$. By Proposition \ref{prop2.1}, $(M, \wtl{g})$ is self-dual. It follows from Theorem $1^{\prime}$ in \cite{ADM} that $\wtl{g}$ must be conformal to one of the following metrics: (1) a K\"ahler meric $g$ with constant holomorphic sectional curvature; (2) a non-flat K\"ahler meric $g$ which has local K\"ahler form $2\sqrt{-1}\dfrac{dz_1\wedge d\bar{z}_1}{(1-|z_1|^2)^2}+2\sqrt{-1}\dfrac{dz_2\wedge d\bar{z}_2}{(1+|z_2|^2)^2}$; (3) the standard Hopf metric $\omega_h$ on an isosceles Hopf surface (see also \cite{CZ}).

With the computation formulae in section 4, we first prove that in case (1) $\wtl{g}$ must be a constant multiple of $g$. Assume $\wtl{g}=e^{2f}g$. As $\wtl{g}$ has pointwise constant holomorphic sectional curvature $\wtl{c}$ with respect to $\wtl{D}^t_s$ and $g$ has constant holomorphic sectional curvature $c$, by (\ref{h22}), 
\begin{align}(e^{2f}\dfrac{\wtl{c}}{2}-\dfrac{c}{2})(\delta_{kl}\delta_{ij}+\delta_{jk}\delta_{il})=&-\dfrac{1}{2}(f_{i\bl}\delta_{jk}+f_{k\bl}\delta_{ij}+f_{i\bj}\delta_{kl}+f_{k\bj}\delta_{il})\notag\\
&+\dfrac{(p-1)^2+s^2}{4}(f_if_{\bj}\delta_{kl}+f_kf_{\bj}\delta_{il}+f_if_{\bl}\delta_{kj}+f_kf_{\bl}\delta_{ij})\notag\\
&-\dfrac{(p-1)^2+s^2}{2}f_rf_{\bar{r}}(\delta_{kl}\delta_{ij}+\delta_{jk}\delta_{il})\label{h224}.
\end{align}
Let $\xi=e^{\mu f}$ with $\mu=-\dfrac{(p-1)^2+s^2}{2}$ and $\la=-\dfrac{1}{2}(\mu\xi (e^{2f}\wtl{c}-c)-\dfrac{2}{\xi}\xi_r\xi_{\bar{r}})$. Then (\ref{h224}) becomes \begin{align*}
2\la(\delta_{kl}\delta_{ij}+\delta_{jk}\delta_{il})=\xi_{i\bl}\delta_{jk}+\xi_{k\bl}\delta_{ij}+\xi_{i\bj}\delta_{kl}+\xi_{k\bj}\delta_{il}.
\end{align*}
By choosing different indices, we get
\begin{align}
\xi_{i\bar{j}}=\la \delta_{ij}.\label{6.1}
\end{align} Since $g$ is K\"ahler ($R_{\bi\bj**}=0$), $\la_{\bj}=\xi_{i\bi \bj}=\xi_{i\bj\bi}=0$. So $\la$ is a constant. Contracting (\ref{6.1}) by $g$, we have $\Delta\xi=2\la$ on $M$. As $M$ is compact, by the maximum principle, $\la=0$ and $\xi$ is a constant. So $\wtl{g}$ is a constant multiple of $g$ and is K\"ahler.

Next, we show that case (2) is not possible, following \cite{CZ}. Denote $g$ to be the K\"ahler metric with local K\"ahler form $\omega=2\sqrt{-1}\dfrac{dz_1\wedge d\bar{z}_1}{(1-|z_1|^2)^2}+2\sqrt{-1}\dfrac{dz_2\wedge d\bar{z}_2}{(1+|z_2|^2)^2}$, where $(z_1,z_2)$ is a local holomorphic coordinate. Then $g$ is the product metric of the Bergman metric on $\mathbb D$ and the Fubini-Study metric on $\mathbb P^1$. Let $\{e_1, e_2\}$ be a unitary frame of $g$ which are parallel to $\frac{\pa}{\pa z_1}, \frac{\pa}{\pa z_2}$ respectively. Then $R_{1\bar{1}1\bar{1}}=-1, R_{2\bar{2}2\bar{2}}=1$ and all other components $R_{i\bar{j}k\bar{l}}$ of $g$ are zero. Assume that $\wtl{g}=e^{2f}g$ for some global function $f$ and $\wtl{g}$ has pointwise constant holomorphic sectional curvature $\wtl{c}$ with respect to $\wtl{D}^t_s$. Letting $(k\bar{l}i\bar{j})$ be $(1\bar{1}1\bar{1}), (2\bar{2}2\bar{2})$ and $(1\bar{1}1\bar{2})$ in (\ref{h22}), we get
\begin{align} e^{2f}\wtl{c}+1+((1-p)^2+s^2)f_rf_{\bar{r}}&=-2f_{1\bar{1}}+((1-p)^2+s^2)f_1f_{\bar{1}},\notag \\
e^{2f}\wtl{c}-1+((1-p)^2+s^2)f_rf_{\bar{r}}&=-2f_{2\bar{2}}+((1-p)^2+s^2)f_2f_{\bar{2}},\notag \\
0&=-2f_{1\bar{2}}+((1-p)^2+s^2)f_1f_{\bar{2}}. \label{6.01}
\end{align}
If $(t,s)=(1,0)$, then $(1-p)^2+s^2=0$. The above equations become 
$$ f_{1\bar{1}}=-\dfrac{1}{2}(e^{2f}\wtl{c}+1),\hspace{.5cm}
f_{2\bar{2}}=-\dfrac{1}{2}(e^{2f}\wtl{c}-1),\hspace{.5cm}
f_{1\bar{2}}=f_{2\bar{1}}=0.$$
As $g$ is K\"ahler, $f_{1\bar{1}\bar{2}}=f_{1\bar{2}\bar{1}}=0, f_{2\bar{2}\bar{1}}=f_{2\bar{1}\bar{2}}=0$. So $\bar{e}_1(e^{2f}\wtl{c})=\bar{e}_2(e^{2f}\wtl{c})=0$. Since $e^{2f}\wtl{c}$ is a global real function,  $e^{2f}\wtl{c}$ must be a constant. Therefore, $\Delta f$ is a constant and by the maximum principle, $f$ is a constant. Then $0=e^{2f}\wtl{c}+1=e^{2f}\wtl{c}-1$, which is a contradiction.\\
If $(t,s)\neq (1,0)$, then $(1-p)^2+s^2\neq 0$. Let $\mu=-\dfrac{(1-p)^2+s^2}{2}$, $\xi=e^{\mu f}$ and $\la=-\dfrac{\mu \xi}{2}(e^{2f}\wtl{c}-2\mu f_rf_{\bar{r}})$. Then (\ref{6.01}) becomes
\begin{align*} \xi_{1\bar{1}}=\la+\dfrac{\mu \xi}{2}, \hspace{1cm} \xi_{2\bar{2}}=\la-\dfrac{\mu \xi}{2},\hspace{1cm} \xi_{1\bar{2}}=\xi_{2\bar{1}}=0.
\end{align*}
Let $\alpha=\la-\dfrac{\mu \xi}{2}, \beta=\la+\dfrac{\mu \xi}{2}$, which are two real global functions. Then $\alpha_{\bar{1}}=0, \beta_{\bar{2}}=0$. So $\alpha$ depends only on $z_2$ and $\beta$ depends on $z_1$. Then $\Delta \alpha=\alpha_{2\bar{2}}=(\beta-\mu \xi)_{2\bar{2}}=-\mu \alpha$.  As $\mu<0$, by the maximum principle, $\alpha=0$. So $\Delta \xi=\xi_{1\bar{1}}+\xi_{2\bar{2}}=\mu \xi$. Then by the maximum principle, the minimum of $\xi$ is nonpositive, which is a contradiction. In conclusion, the case (2) is impossible. 

Case (3) is just contained in the situation of Proposition \ref{prop4.3}. Then we deduce that $(1-t+ts)^2+s^2=4$ and $\wtl{g}$ is an admissible metric. The proof of Theorem \ref{thm2} is finished.

Last, we prove Corollary \ref{cor1.8}. Using the adjusted notation, assume that $\wtl{g}$ has pointwise constant holomorphic sectional curvature with respect to $\wtl{D}^t_s$ for some $(t,s)\in \mathbb R^2$ and is non-K\"ahler. Then by Theorem \ref{thm2}, $\wtl{g}$ must be an admissible metric on an isosceles Hopf surface with $(t,s)$ satisfying $(1-t+ts)^2+s^2=4$. In this case, by Proposition \ref{prop4.3}, the holomorphic sectional curvature is $H=-\dfrac{1}{\xi_A}(4\ {}^{t}zA\bar{A}\bar{z}+{}^{t}zAz+\overline{{}^{t}zAz})$, where $\xi_A=|z|^2+{}^{t}zAz+\overline{{}^{t}zAz}$ and $A$ is symmetric and satisfies $A\bar{A}<\dfrac{1}{4}I_n$. Diagonalizing $A$ with a unitary matrix $U$ (see page 8 in \cite{CZ}), direct calculation shows that $H$ can not be a positive function or a negative function. So Corollary \ref{cor1.8} is proved.

\end{document}